\documentclass{amsart}
\usepackage{amssymb}
\usepackage{array}
\usepackage{multirow}

\DeclareMathOperator{\Aut}{Aut}
\DeclareMathOperator{\Int}{Int}
\DeclareMathOperator{\im}{im}
\DeclareMathOperator{\GL}{GL}

\newcommand{\inv}{^{-1}}
\newcommand{\phm}{\phantom{-}} 
\newcommand{\NN}{\mathcal{N}}
\newcommand{\F}{\mathbf{F}}
\newcommand{\N}{\mathbf{N}}
\newcommand{\Z}{\mathbf{Z}}
\newcommand{\greatintfrac}[2]{{\left\lfloor\frac{#1}{#2}\right\rfloor}}
\newcommand{\bgreatintfrac}[2]{{\bigl\lfloor\frac{#1}{#2}\bigr\rfloor}}

\newcommand{\exerdoit}{\begin{exercise}Do it.\end{exercise}}

\newcommand{\cvectwo}[2]{\Bigl(\begin{matrix}#1\\#2\end{matrix}\Bigr)}
\newcommand{\cvecthree}[3]{
	\Bigl(\begin{smallmatrix}#1\\#2\\#3\end{smallmatrix}\Bigr)}

\newcommand{\scvectwo}[2]{
	\bigl(\begin{smallmatrix}#1\\#2\end{smallmatrix}\bigr)}

\newcommand{\smattwo}[4]{
	\bigl(
	\begin{smallmatrix}#1&#2\\#3&#4\end{smallmatrix}
	\bigr)}

\newcommand{\mattwo}[4]{\Bigl(\begin{matrix}#1&#2\\#3&#4\end{matrix}\Bigr)}
\newcommand{\matthree}[9]{
	\Bigl(
	\begin{smallmatrix}
		#1&#2&#3\\
		#4&#5&#6\\
		#7&#8&#9
	\end{smallmatrix}
	\Bigr)
	}

\newcommand{\matjordanone}{\matthree110010001}
\newcommand{\matjordantwo}{\matthree110011001}

\newtheorem{thm}{Theorem}
\newtheorem*{mainthm}{Main Theorem}
\newtheorem{prop}[thm]{Proposition}
\newtheorem{lemma}[thm]{Lemma}
\newtheorem{cor}[thm]{Corollary}

{
\theoremstyle{definition}
\newtheorem{defn}[thm]{Definition}
\newtheorem{notation}[thm]{Notation}
\newtheorem*{exercise}{Exercise}
}

\title{Groups of order $p^4$ made less difficult}
\author{Jeffrey D. Adler}
\address{Department of Theoretical and Applied Mathematics \\
The University of Akron \\ Akron, OH  44325-4002}
\curraddr[Adler]{Department of Mathematics and Statistics \\
American University \\
Washington, DC 20016-8050}
\email[Corresponding author]{jadler@american.edu}
\author{Michael Garlow}
\email{garlow@uakron.edu}
\author{Ethel R. Wheland}
\email{wheland@uakron.edu}
\date{20 July, 2006.
We inserted, deleted, or changed about 30 words in October, 2016.}

\begin{document}

\begin{abstract}
Using only undergraduate-level methods, we classify all groups
of order $p^4$, where $p$ is an odd prime.
\end{abstract}

\maketitle

\addtocounter{section}{-1}
\section{Introduction}
Marcel Wild recently provided a simple classification
\cite{Wild} of the $14$ groups of order $16$.
The virtue of his presentation is that it relies only
on elementary methods,
except for his use of the Cyclic Extension Theorem
(see Theorem \ref{thm:cyclic-extn} below),
for which Wild
suggested that there should be an elementary proof.
The present article has two purposes:
to provide such a proof, thus putting all of Wild's
machinery onto an elementary footing;
and to extend this machinery enough that one
could use it to classify
the groups of order $p^4$ for any prime $p$
without a lot of effort.
We perform this classification when $p$ is odd.
(There are $15$ such groups.)
For completeness, we also comment on the $p=2$ case,
though of course it is already covered by Wild.
Our methods and results lead to several
interesting projects and problems for students.

The cost of the generality of $p$
is that we must depend on some basic linear algebra
in order to avoid an explosion of \textit{ad hoc} calculations.
More specifically,
of all the results and tools that we use, the following
are the most advanced:
\begin{itemize}
\item
exactly half of the nonzero elements of a prime field of odd order are
squares;
\item
every nontrivial finite $p$-group has a normal subgroup of index $p$;
\item
every nontrivial finite $p$-group has a nontrivial center;
\item
basic manipulation of matrices;
\item
theory of Jordan canonical forms (Lemma~\ref{lem:jordan});
\item
Fundamental Theorem of Finite Abelian Groups.
\end{itemize}
Some of these could be dispensed with,
but at the cost of more computation.

In \S\ref{sec:cyclic},
we recall the notion of an \emph{extension type}
(which we learned from \cite{Wild});
show that every extension type determines a group
(this is the Cyclic Extension Theorem);
and explore (in \S\ref{sec:equivalence})
when two extension types determine isomorphic groups.
In \S\ref{sec:construction},
we first discuss the abelian case, and then
begin the construction of all nonabelian groups of order $p^4$
by cooking up a collection of extension types that
together must describe all such groups.
To keep this collection small, we make heavy use of 
the results of \S\ref{sec:equivalence}.
In \S\ref{sec:classification},
we complete the classification by determining precisely 
which extension types from \S\ref{sec:construction}
yield isomorphic groups.
Finally, in \S\ref{sec:p=2} we comment on the $p=2$ case.

The classification of the groups of order $p^4$, and much else,
was known to H\"older \cite{Holder}
and Young \cite{Young},
and is covered in
\S\S112--118
of the textbook of Burnside \cite{Burnside}.
One can go much farther with the help of computers.
See \cite{besch-eick-obrien:millenium} for a history
of the classification of small groups,
and 
\cite{obrien-vlee:p7}
for an example of recent progress.

\textsc{Notation:}
Let $p$ denote an odd prime.
For a natural number $n\in \N$,
$C_n$ will denote the cyclic group with $n$ elements.
If $x$ is an element of a group $G$, then
$\Int(x)$ will denote the automorphism of $G$ that
sends each element $g\in G$ to $xgx\inv$.
If $\tau$ is in the group $\Aut(G)$ of automorphisms of $G$,
then $G^\tau$
will denote the set of fixed points of $\tau$ in $G$:
$G^\tau =\{g\in G \colon \tau(g)=g\}$.
We let $|G|$ denote the order of $G$.
Let $\F_p$ denote the field with $p$ elements,
and $\GL_m(\F_p)$ the group of invertible
$m$-by-$m$ matrices with entries in $\F_p$.
Fix a nonsquare $\varepsilon$ modulo $p$.

\textsc{Acknowledgements:}
This paper is a refinement and generalization of the second-named author's
master's thesis~\cite{Garlow},
written under the supervision of the other two authors.
It is a pleasure to thank
Jeffrey Riedl for his careful reading of the thesis;
Marcel Wild for encouragement;
and the Department of Theoretical and Applied Mathematics
at The University of Akron for financial support.

\section{Cyclic Extensions}
\label{sec:cyclic}

\subsection{Motivation}
\label{sec:cyclic-motivation}

Let $G$ be a finite group.
Consider a normal subgroup $N\vartriangleleft G$,
where $G/N$ is cyclic of order $n$.
Choose any $a\in G\smallsetminus N$ such that the coset $Na$ generates $G/N$.
Let $v=a^{n}\in N$, and let $\tau\in \Aut(N)$ act via conjugation by $a$.
Thus,
\[
\tau(v)=aa^{n}a\inv=a^{n}=v \, .
\]
In other words, $\tau$ fixes $v$.
Also, for $x\in N$,
\[
\tau^{n}(x)=a^{n}xa^{-n}=vxv\inv \, . 
\]
That is, $\tau^n = \Int(v)$.
In particular, if $N$ is an
abelian group, then $\tau^n$ is the identity automorphism.

\begin{defn}
\label{defn:extn-type}
An \emph{extension type} for a group $N$ is
a quadruple $(N,n,\tau,v)$, where $n\in \N$, $v\in N$,
and $\tau\in \Aut(N)$ is such that $\tau^n = \Int(v)$.
\end{defn}

Following \cite{Wild},
notice that this definition is stated without mention of a group $G$.
However, starting with a group $G$, choices of normal subgroup $N$
(having cyclic factor group of order $n$)
and element $a\in G\smallsetminus N$ (such that $aN$ generates $G/N$)
will determine an extension type
$(N,n,\tau,v)$ as above.
Moreover, the
extension type determines $G$ up to isomorphism.
To see this, note that since $N$ and $a$ generate $G$,
every element of $G$ has the form $xa^i$ for some $x\in N$
and $0\leq i < n$.
Thus, in order to determine the operation table of $G$,
we only need to know how to multiply two such elements $xa^i$
and $ya^j$ to obtain a third element in the same form.
Now note that
$$
(xa^i)(ya^j)
=x (a^i y a^{-i}) a^i a^j
= (x \tau^i (y)) a^{i+j},
$$
which we will rewrite as 
$(x \tau^i (y)v ) a^{i+j-n}$
if $i+j\geq n$.

\subsection{Using extension types to construct groups}

We have just seen that each extension type determines at most
one group up to isomorphism.
The next result guarantees that
each extension type does indeed determine a group.

\begin{thm}[Cyclic Extension Theorem]
\label{thm:cyclic-extn}
Each extension type
$(N,n,\tau,v)$ determines a group.
\end{thm}

\begin{proof}
Following \cite{Wild},
let $G$ be the set of ordered pairs $(x,a^{i})$ for $x\in N$
and $i\in\{0,1,...,n-1\}$,
and define a binary operation $\ast$ on $G$ by
$$
(x,a^{i})\ast(y,a^{j})=
\begin{cases}
(x \tau^i(y), a^{i+j}) & \text{if $i+j<n$}, \\
(x \tau^i(y)v, a^{i+j-n}) & \text{if $i+j \geq n$}.
\end{cases}
$$
We will find it
convenient to use the following equivalent definition:
\[
(x,a^{i})\ast(y,a^{j})=
\left(  x\tau^{i}(y)v^{\greatintfrac{i+j}{n}}
,a^{i+j-n\greatintfrac{i+j}{n} }\right)
.
\]
We must show that $\ast$ satisfies the three axioms of a group operation.
If associativity is known, then
it is easy enough to check that
$(e,a^{0})$ is the identity
(where $e$ is the identity in $N$)
and that
$(\tau^{-i}\left[  (vx)\inv \right]  ,a^{n-i})$
is the inverse of $(x,a^{i})$.
Therefore, it only remains to
show that $\ast$ is associative.
We will use the elementary fact
that for all integers $m$ and $r$,
$v^{r}\tau^{m-nr}(z)=\tau^{m}(z)v^{r}$.
Consider
\begin{align*}
&  \left[  (x,a^{i})\ast(y,a^{j})\right]  \ast(z,a^{k}) \\
&  =\left(  x\tau^{i}(y)v^{\greatintfrac{i+j}{n}}
\:,\:
a^{i+j-n\greatintfrac{i+j}{n}}\right)  \ast (z,a^{k}) \\
&  =\Bigl(  x\tau^{i}(y)v^{\greatintfrac{i+j}{n}}
\tau^{i+j-n\greatintfrac{i+j}{n}}(z)
v^{\bgreatintfrac{i+j+k-n\greatintfrac{i+j}{n}}{n}}
\:,\:
a^{i+j-n\greatintfrac{i+j}{n}+k-n
	\bgreatintfrac{i+j+k-n\greatintfrac{i+j}{n}}{n}}
\Bigr)   \\
&  =\Bigl(  x\tau^{i}(y)\tau^{i+j}(z)v^{\greatintfrac{i+j}{n}}
v^{\bgreatintfrac{i+j+k-n\greatintfrac{i+j}{n}}{n}}
\:,\:
a^{i+j+k-n\greatintfrac{i+j+k}{n}}\Bigr)   \\
&  =\left(  x\tau^{i}(y)\tau^{i+j}(z)v^{\greatintfrac{i+j+k}{n}}
\:,\:
a^{i+j+k-n\greatintfrac{i+j+k}{n}}\right)   \\
&  =\left(  x\tau^{i}\left[  y\tau^{j}(z)\right] v^{\greatintfrac{i+j+k}{n}}
\:,\:
a^{i+j+k-n\greatintfrac{i+j+k}{n} }\right)  \,,
\end{align*}
which is equal to
$(x,a^{i})\ast\left[  (y,a^{j})\ast(z,a^{k})\right]$
by similar reasoning.
\end{proof}

We will often write $a$ for $(e,a)$ and $x$ for $(x,a^{0})$.

\subsection{Equivalence of extension types}
\label{sec:equivalence}

Recall (from the introduction) that every group of order $p^4$ has 
a normal subgroup of order $p^3$.
Thus,
in order to construct all groups of order $p^4$, it would be sufficient to
construct all extension types $(N,p,\tau,v)$, where $N$ is a group of order
$p^3$.
However, the number of such extension types is huge.
We present some
techniques for identifying when two extension types determine
isomorphic groups.

\begin{defn}
\label{defn:equiv}
Two extension types are \emph{equivalent} if
they determine isomorphic groups.
\end{defn}

Note that in \cite{Wild}, equivalence has the following meaning.

\begin{defn}
The extension types $(N,n,\tau,v)$ and
$(N',n,\sigma,\omega)$ are \emph{conjugate} if there is an
isomorphism $\phi:N\longrightarrow N'$ such that $\sigma=\phi\circ
\tau\circ\phi\inv$ and $\omega=\phi(v)$.
\end{defn}

\begin{lemma}
\label{lem:conj-are-equiv}
Conjugate extension types are equivalent.
\end{lemma}

\begin{proof}
This is Lemma 1(a) of \cite{Wild}.
\end{proof}

\begin{lemma}
\label{lem:v-powers-equiv}
Let $N$ be finite abelian, let $n\in \N$,
and let $i\in \Z$ be prime to $|N|$.
Then
the extension types $(N,n,\tau,v)$ and $(N,n,\tau,v^i)$ are conjugate.
\end{lemma}

\begin{proof}
By hypothesis, the map
$\phi:N\longrightarrow N$ defined by
$\phi(x)=x^i$
has trivial kernel, and is thus an automorphism.
For all $x\in N$,
$$
\phi(\tau(x))  =(\tau(x))^i=\tau(x^i)=\tau(\phi(x)).
$$
That is, $\phi\tau = \tau\phi$,
so $\phi\tau\phi\inv = \tau$.
Since $\phi(v)=v^i$, the result follows.
\end{proof}

\begin{notation}
Given $\tau\in\Aut(N)$ and $n\in\N$,
define a function
$\NN_{\tau,n}:N\longrightarrow N$ by
\[
\NN_{\tau,n}(x)=x\tau(x)\tau^{2}(x)\cdots\tau^{n-1}(x).
\]
We will elide the subscripts $\tau$ and $n$ when they are understood
from the context.
\end{notation}

The usefulness of the function $\NN_{\tau,n}$ comes from the following
lemma.

\begin{lemma}
\label{lem:power-norm}
Suppose $N \vartriangleleft G$ and
$\tau\in \Aut(N)$ acts via conjugation by $a \in G$.
Then for all $x\in N$,
$(xa)^n = \NN_{\tau,n}(x) a^n$.
\end{lemma}

\begin{proof}
We have
\begin{align*}
(xa)^{n}
&  =\left(  xa\right)  \left(  xa\inv a^{2}\right)  \left(  xa^{-2}
a^{3}\right)  \cdots(xa^{-(n-1)}a^{n})\\
&  =x(axa\inv)(a^{2}xa^{-2})\cdots(a^{n-1}xa^{-(n-1)})a^{n}\\
&  =x\tau(x)\tau^{2}(x)\cdots\tau^{n-1}(x) a^n \\
&  =\NN_{\tau,n}(x) a^n \, .
\qedhere
\end{align*}
\end{proof}

Several remaining results
(Lemmata~\ref{lem:change-coset-rep},
\ref{lem:change-coset},
and
\ref{lem:special-equiv},
and Proposition~\ref{prop:no-cyclic})
assert that one extension
type is equivalent to another.
Their proofs all follow the same outline.
Given an extension type $(N,n,\tau,v)$, construct a group
$G$ as in the proof of Theorem~\ref{thm:cyclic-extn}.
Choose a normal subgroup $N'$ of $G$ (often but not always
equal to $N$), and an element $a'\in G\smallsetminus N'$ such
that $a'N'$ generates $G/N'$.
As in \S\ref{sec:cyclic-motivation},
we obtain an extension type $(N',n,\tau',v')$.
Since this type is realized by $G$,
it must be equivalent to $(N,n,\tau,v)$.

\begin{lemma}
\label{lem:change-coset-rep}
For $x\in N$, the
extension types $(N,n,\tau,v)$ and $(N,n,\Int(x)\tau,\NN_{\tau,n}(x)v)$
are equivalent.
\end{lemma}

\begin{proof}
Let $G$ be a group that realizes $(N,n,\tau,v)$
and constructed as in the proof of Theorem~\ref{thm:cyclic-extn}.
We will show that $G$ also realizes $(N,n,\tau,\NN(x)v)$.
Let $a'=xa$, and
let $\tau'\in \Aut(N)$ act via conjugation by $a'$.
For $y\in N$, consider
\[
\tau'(y)=a'y\left(  a'\right)  \inv
=\left(  xa\right)  y\left(  a\inv x\inv \right)
=x(aya\inv)x\inv 
=x\tau(y)x\inv
=(\Int(x)\tau)(y) \, ,
\]
so $\tau'=\Int(x)\tau$.
From Lemma~\ref{lem:power-norm},
$(a')^n  =\NN(x)v$.
\end{proof}

\begin{cor}
\label{cor:change-coset-rep}
If $N$ is abelian and $x\in N$, then the
extension types $(N,n,\tau,v)$ and $(N,n,\tau,\NN_{\tau,n}(x)v)$
are equivalent.
\end{cor}

\begin{lemma}
\label{lem:change-coset}
If $i\in \Z$ is prime to $n$, then the
extension types $(N,n,\tau,v)$ and $(N,n,\tau^{i},v^{i})$
are equivalent.
\end{lemma}

\begin{proof}
Let $G$ be a group that realizes $(N,n,\tau,v)$.
This
extension type is determined by choices of $N$ and $a\in G\smallsetminus N$.
One could
just as easily choose $a^{i}\in G\smallsetminus N$
and the resulting extension type would be
$(N,n,\tau^{i},v^{i})$.
\end{proof}

\section{Construction of all groups of order $p^4$}
\label{sec:construction}
\subsection{Dispensing with the abelian case}
In order to construct the groups of order $p^4$,
we could construct, up to equivalence, all extension types
$(N,p,\tau,v)$, where $N$ is a group of order $p^3$.
However,
from the Fundamental Theorem of Finite Abelian Groups
(which follows easily from Theorem 5.3 of \cite{Dummit}),
up to isomorphism the abelian groups are given by the following list:
\[
C_{p^4}, \:\:
C_{p^3}\times C_p, \:\:
C_{p^2} \times C_{p^2}, \:\:
C_{p^2}\times C_p \times C_p, \:\:
C_p \times C_p \times C_p \times C_p.
\]
Therefore, we may concentrate on the nonabelian case.
This allows us to assume that the automorphism $\tau$ is nontrivial.

\subsection{Subgroups $N$}

A corollary of Sylow's Theorem ensures that a group of order $p^4$
has a subgroup of order $p^3$.
Moreover, all subgroups of order $p^3$ are normal.
There are five groups of order $p^3$,
three of which are abelian and two of which are nonabelian.
When constructing extension types, $N$ must
come from this list of five groups.
It is desirable to cut this list down.

\begin{prop}
\label{prop:ab-subgp}
Every group $G$ of order $p^{4}$ has an abelian subgroup of
order at least $p^{3}$.
\end{prop}

\begin{proof}
Let $Z$ denote the center of $G$.
If $\left\vert Z\right\vert \geq p^{3}$,
then we are done.
Since $p$-groups have
nontrivial centers, we may assume that $\left\vert Z\right\vert =p$ or
$\left\vert Z\right\vert =p^{2}$.

We claim that $G$ has a normal
subgroup $H$ of order $p^{2}$.
If $\left\vert Z\right\vert =p^{2}$,
then this is obvious,
so suppose that $\left\vert Z\right\vert =p$.
By the Lattice Isomorphism Theorem
(Theorem 3.20 in \cite{Dummit}),
it is enough to find a normal subgroup of $G/Z$ of order $p$.
Since $G/Z$ is a $p$-group, its center has order at least $p$,
and thus contains a normal (in $G/Z$) subgroup of order $p$,
so the claim is proved.

Now define a
homomorphism $\Phi:G\longrightarrow \Aut(H)$ by
\[
\left[  \Phi(g)\right]  (h)=ghg\inv \, .
\]
If $g,h\in H$, then
$[\Phi(g)] (h)=h$.
This implies $H\subseteq\ker(\Phi)$.

We wish to show that this containment is strict.
Suppose for a contradiction
that $\ker(\Phi)=H$.
Then
$\left\vert \ker(\Phi)\right\vert =\left\vert
H\right\vert =p^{2}$,
which gives $\left\vert G/\ker(\Phi)\right\vert =p^{2}$.
By the First Isomorphism Theorem (Theorem 3.16 in \cite{Dummit}),
$\Phi$ corresponds to a one-to-one map
from $G/\ker(\Phi)$ to $\Aut(H)$.
Therefore, $\left\vert G/\ker
(\Phi)\right\vert $ must divide $\left\vert \Aut(H)\right\vert $.
However,
since $H=\ker(\Phi)$ is isomorphic to either $C_{p^{2}}$ or
$C_p\times C_p$, we have that $\left\vert \Aut(H)\right\vert = p^2-p$
or $(p^2-1)(p^2-p)$,
neither
of which is divisible by $p^{2}$, and so we have a contradiction.

Thus, we may pick an element
$g\in \ker(\Phi)\smallsetminus H$.
Since $g$ must commute with all elements of $H$,
the group generated by $H$ and $g$
is an abelian subgroup of $G$ of order at least $p^3$.
\end{proof}

Therefore, in constructing extension types,
we may assume that $N$ is one of the three abelian groups
of order $p^3$.
However, the following result shows that we
need not consider the case where $N$ is cyclic.

\begin{prop}
\label{prop:no-cyclic}
If a nonabelian group $G$ of order $p^4$ contains a subgroup
isomorphic to $C_{p^3}$,
then $G$ also contains a subgroup isomorphic to $C_{p^2}\times C_{p}$.
\end{prop}

\begin{proof}
Let $H=\langle h\rangle$ be a cyclic subgroup of
$G$ of order $p^3$.
Choose $a\in G\smallsetminus H$,
define $v=a^{3}$ and let $\tau\in \Aut(H)$ act via
conjugation by $a$.
Then $G=\langle h,a\rangle$.
It remains 
to show that $G$ contains a subgroup isomorphic to $C_{p^2}\times C_{p}$.
Since $H$ is cyclic and has order $p^3$,
every automorphism of $H$ will be of the form $x\longmapsto x^{m}$,
where $m$ and $p^3$ are
relatively prime.
Thus, $\Aut(H)$ is cyclic of order $\phi(p^3)$, where $\phi$
is the Euler function.
Since $p$ divides $\phi(p^3) = p^2(p-1)$,
$\Aut(H)$ has $\phi(p)=p-1$ elements
of order $p$.
By Lemma~\ref{lem:v-powers-equiv},
we only need to consider one of them.
Since $p^2+1 \not\equiv 1 \bmod p^3$ but
$$
(p^2+1)^p = p^{2p} + \cdots + \tbinom{p}{1} p^2 + 1 \equiv 1 \bmod p^3,
$$
we have that
$x \longmapsto x^{p^2+1}$ is an automorphism of order $p$.
So assume without loss of
generality that $\tau(x)=x^{p^2+1}$.
Let $H'$ be the subgroup of $H$
generated by $h^{p}$.
This is a cyclic subgroup of order $p^2$, and it commutes with
$a$ since $\tau$ fixes $h^{p}$.
To see this, first notice that
\[
\tau(h^{p})=h^{(p^2+1)p}=h^{p}.
\]
Also, $\tau(h^{p})=ah^{p}a\inv$, which gives that $ah^{p}=h^{p}a$.
Thus, it 
suffices to show the existence of an element $x\in G$ of order $p$ such that
$x \notin H'$
and
$xh^{p}  =h^{p}x$,
since then
$\left\langle h^{p},x\right\rangle \cong C_{p^2}\times C_{p}$.
Let $a'=ah^{r}$, where
$r$ is to be determined.
Clearly, $ah^{r}\in G\smallsetminus H$.
So, for $n\in H$, consider
\begin{align*}
a'n(a')\inv &  =(ah^{r})n(h^{-r}a\inv )\\
&  =a\left(  h^{r}nh^{-r}\right)  a\inv\\
&  =a(nh^{r}h^{-r})a\inv\\
&  =ana\inv\\
&  =\tau(n) \, .
\end{align*}
Thus, $a'$ commutes with $H'$, just as $a$ does.
Lastly, it
suffices to show that we can choose $r$ such that $a'$ has order $p$.
From Lemma~\ref{lem:power-norm},
\[
(a')^{p}=\NN(h^{r})v \, .
\]
Thus, we must show that
$\NN(h^{r})=v\inv$
for an appropriate choice of $r$.
Consider
\[
\NN(h^{r})
=\prod_{i=0}^{p-1} \tau^i(h^{r})
=h^{r \sum_i (p^2+1)^i} \, .
\]
Reducing the exponent modulo $p^3$, we obtain
\[
r \sum_{i=0}^{p-1} (p^2+1)^i
=
r \frac{(p^2+1)^p - 1}{(p^2+1)- 1}
\equiv rp \bmod p^3.
\]
That is, $\NN(h^r) = h^{rp}$.
Thus, we must show that $v=h^{-rp}$ for some $r$.
Equivalently, $v\in H'$.
However, $v$ is fixed by $\tau$, and $H'$ is the set of
all fixed points of $\tau$.
\end{proof}

Together, Propositions~\ref{prop:ab-subgp} and~\ref{prop:no-cyclic}
produce a powerful result:
Every
nonabelian group of order $p^4$ contains a subgroup isomorphic to either
$C_{p^2}\times C_{p}$ or $C_{p}\times C_{p}\times C_{p}$.
In particular, it
is sufficient to consider only these two choices of $N$
when constructing extension 
types.
Conveniently, much is known about the automorphism groups of both 
of these groups.

\subsection{Automorphisms $\tau$ of $N$}

For $N=C_{p^2}\times C_{p}$ and $N=C_{p}\times C_{p}\times C_{p}$,
we must find
enough automorphisms $\tau$ of $N$ to construct, up to equivalence, all
extension types of the form $(N,p,\tau,v)$.
The following lemma will be useful.

\begin{lemma}
\label{lem:jordan}
Let $m$ be a positive integer.
Every element of $\GL_m(\F_p)$ of order $p$
is conjugate over $\F_p$ to a matrix in Jordan canonical form
with ones along the diagonal.
\end{lemma}

\begin{proof}
Such an element $A$ is a root of the $\F_p$-polynomial $X^p-1$,
which equals $(X-1)^p$, and thus splits over $\F_p$.
Now apply Theorem~12.23 in \cite{Dummit}.
\end{proof}

Note that one could use Sylow theory to
produce a more elementary proof in the cases
where $m= 2$ or $3$, which are the only ones we will need.

\exerdoit

Now we begin our study of the automorphisms
of order $p$ of 
$N=C_{p^2}\times C_{p}$.
Let $x$ be a generator of $C_{p^2}$ and
$y$ be a generator of $C_{p}$.
Then every automorphism of $N$ has the form
\[
\left\{
\begin{aligned}
x & \longmapsto x^{a}y^{b}\\
y & \longmapsto x^{c}y^{d}
\end{aligned}
\right. 
\]
for $a$ and $c$ integers modulo $p^2$, and $b$ and $d$ integers modulo $p$.
This automorphism
can conveniently be represented as the matrix
$\smattwo{a}{c}{b}{d}$.
Note that composition of automorphisms is
compatible with matrix multiplication.
By Lemma~\ref{lem:conj-are-equiv},
we only need to consider automorphisms up to conjugacy in
$\Aut(C_{p^2}\times C_{p})$.
Consider the homomorphism from $\Aut(C_{p^2}\times C_{p})$
into $\GL_2(\F_p)$ that maps the matrix representation 
of an automorphism in $\Aut(C_{p^2}\times C_{p})$
to a matrix in $\GL_2(\F_p)$ by reducing the top
row modulo $p$.
By Lemma~\ref{lem:jordan},
the image in $\GL_2(\F_p)$ of an automorphism of order $p$
in $\Aut(C_{p^2}\times C_{p})$ is conjugate to either
$\smattwo1001$
or
$\smattwo1011$.
Therefore,
we may assume our matrix has the form
$\smattwo{1+ps}{pr}01$
or
$\smattwo{1+ps}{pr}11$,
where $r,s\in\{0,\ldots,p-1\}$.
Moreover, since $\tau$ is not the identity automorphism,
we may assume in a matrix of the former type that $r$ and $s$ are not both zero.

\begin{lemma}
\label{lem:tau-choices2}
If $N=C_{p^2}\times C_{p}$, then we only need to consider
the automorphisms of $N$ represented by the following matrices:
\[
\mattwo1p01,
\mattwo{1+p}001,
\mattwo1011,
\mattwo1p11,
\mattwo1{\varepsilon p}11
\, .
\]
\end{lemma}

\begin{proof}
Recall that in our matrix computations, the first row
of each matrix is taken modulo $p^2$ and the second
is taken modulo $p$.
We will use the following straightforward calculations:
\begin{align}
\tag{$0$-conj}
\label{eqn:conj-zero}
\mattwo1011
\mattwo\alpha\beta01
\mattwo1011 \inv
&  =
\mattwo{\alpha-\beta}{\beta}01
\\
\label{eqn:power-zero}
\tag{$0$-pow}
{\mattwo{1+p}001}^s
& =
\mattwo{1+sp}001
\\
\label{eqn:conj-one}
\tag{$1$-conj}
\mattwo1{-p}0{\phm1}
\mattwo\alpha\beta11
\mattwo1{-p}0{\phm1} \inv
&  =
\mattwo{\alpha-p}{\beta}11
\\
\label{eqn:power-one}
\tag{$1$-pow}
{\mattwo1{rp}11}^q
& =
\mattwo{1+\binom{q}{2}rp}{qrp}q1
\, .
\end{align}

First, consider a matrix of the form
$\smattwo{1+sp}{rp}01$
for $r,s\in \{0,\ldots,p-1\}$.
If $r\neq0$,
then we may pick $0<t<p$
so that $rt\equiv 1 \bmod p$.
Apply \eqref{eqn:conj-zero} $ts$ times to see that our
matrix is conjugate to
$\smattwo1{rp}01$,
which equals
${\smattwo1{p}01}^r$.
By Lemmata~\ref{lem:conj-are-equiv} and~\ref{lem:change-coset},
we can replace our original matrix by
$\smattwo1p01$.
On the other hand,
if $r=0$, then $s\neq0$, and so
by \eqref{eqn:power-zero} and Lemma~\ref{lem:change-coset},
we may replace our matrix by
$\smattwo{1+p}001$.

Now consider a matrix of the form
$\smattwo{1+sp}{rp}11$.
Applying \eqref{eqn:conj-one} $s$ times,
we see that our matrix is conjugate to
$\smattwo1{rp}11$.
From \eqref{eqn:power-one},
we see that for all $0<q<p$,
the $q$th power of this latter matrix is
$\smattwo{1+\binom{q}{2}rp}{qrp}q1$.
Applying \eqref{eqn:conj-one} $\binom{q}{2}r$ times,
we see that this third
matrix is conjugate to
$\smattwo1{qrp}q1$.
Now note that
$$
\mattwo{q}001{\mattwo1{qrp}q1}{\mattwo{q}001}\inv
= \mattwo1{q^2rp}11.
$$
By Lemmata~\ref{lem:conj-are-equiv} and~\ref{lem:change-coset},
we may thus replace our original matrix by
$\smattwo1{r'p}11$,
where $r'$ is any number that is in the same class as $r$ modulo
squares mod $p$.
Thus, we may assume that $r'$ is $0$, $1$, or $\varepsilon$.
\end{proof}

Next, assume $N=C_{p}\times C_{p}\times C_{p}$.
We will view $N$ as a three-dimensional vector
space over $\F_p$.
By Proposition 4.17(3) in \cite{Dummit},
$\Aut(N)\cong \GL_3(\F_p)$.

\begin{lemma}
\label{lem:tau-choices3}
If $N=C_p \times C_p \times C_p$,
then every automorphism of $N$ of order $p$
is conjugate to one of the following:
$$
\matjordanone,
\matjordantwo \, .
$$
\end{lemma}

\begin{proof}
This follows from Lemma~\ref{lem:jordan}.
\end{proof}

\subsection{Choices for $v$}
For each of the seven choices of $\tau$ that we have identified
in Lemmata~\ref{lem:tau-choices2} and~\ref{lem:tau-choices3},
we wish to identify a set of choices for $v$
so that all pairs $(\tau,v)$ taken together
will be sufficient
to construct all nonabelian groups of order $p^4$.

Recall that $v$ must belong to the set $N^\tau$
of fixed points of $\tau$ in $N$.
For each $\tau$, finding $N^\tau$ is a straightforward matrix
calculation, equivalent to finding $\ker(\tau-I)$, where
$I$ is the identity map on $N$.
From Corollary~\ref{cor:change-coset-rep}, two choices of $v\in N^\tau$
give equivalent extension types if they differ by an element
of the image of $\NN_{\tau,p}$.
Since $\NN_{\tau,p}$
has the matrix representation $I + \tau + \cdots + \tau^{p-1}$,
it is a straightforward matrix computation to find its image.
Note that since $N$ is abelian, $N^\tau$ and $\im(\NN_{\tau,p})$
are groups, and we are thus interested in choosing $v$ from
a set of coset representatives of $N^\tau / \im(\NN_{\tau,p})$.
From Lemma~\ref{lem:v-powers-equiv}, two such representatives
give equivalent extension types if one is a power of the other (of order
prime to $p$); or, equivalently, if they generate the same subgroup
of $N^\tau$.

For all but one of our choices for $\tau$, this reasoning alone
is enough to show that we only need to consider one or two choices
for $v$.
The results are presented in Table~\ref{table:v-choices}.
The remaining case is the one where
$N = C_p\times C_p\times C_p$ and
$\tau = \matjordanone$,
presented in the table on the line marked (\S).
In this case
the table
reflects our
claim that it is sufficent to take $v$ to be trivial.

To prove this claim, we will show that for any nontrivial $v$,
a group $G$ with extension type $(N,p,\tau,v)$
must contain a group isomorphic to $C_{p^2}\times C_p$,
and thus has already been constructed.
It is enough to
find commuting elements $x,y\in G$ such that $x$ has order $p^2$,
$y$ has order $p$,
and $y\notin \langle x\rangle$.

Let $a\in G$ be an element such that $\tau$ acts via conjugation by $a$.
Since $v$ has order $p$, $a$ has order $p^2$.
The set $N^\tau$ of fixed points of
$\tau$ can be viewed as a $2$-dimensional subspace
of the $\F_p$-vector space $N$.
Meanwhile, $\left\langle v\right\rangle$ is a
1-dimensional subspace of $N^\tau$,
so $N^\tau\smallsetminus \langle v \rangle$ is nonempty.
Pick $x=a$ and $y\in N^\tau\smallsetminus \langle v \rangle$.
This proves the claim, and finishes the justification of
Table~\ref{table:v-choices}.

\begin{table}
\newcommand{\gen}[1]{\Bigl\langle #1 \Bigr\rangle}
\newcommand{\set}[1]{\Bigl\{ #1 \Bigr\}}
\newcommand{\RowWithNote}[7][{}]{#1 & #2 & \gen{#3} & #4 & \gen{#5} &
	\set{#6}
	\vphantom{\set{#6}_0^0} & #7 \\}
\newcommand{\row}[6][{}]{\RowWithNote[#1]{#2}{#3}{#4}{#5}{#6}{}}
\newcommand{\splitrow}[8]{
	&
	\multirow{2}{*}[-1ex]{$#1$} &
	\multirow{2}{*}{
		$\phantom{\Biggr . \Biggl \} } #2 \Biggl\{ \Biggr . $
		}  &
	#3 & \gen{#5} & \set{#7} & (p=3) \\
	& & & #4 & \gen{#6} & \set{#8}\vphantom{\set{#8}^0} & (p>3) \\[2ex]
	}
\newcommand{\matzero}{\matthree000000000}
$$
\begin{array}{c c | c @{} c c l @{} c}
&\multicolumn{1}{c}{\tau} & N^\tau & \NN & \im(\NN)& \text{$v$ choices}\\
\cline{2-6}
\row{\mattwo1p01}{\cvectwo10}{\mattwo p000}{\cvectwo p0}{\cvectwo00,\cvectwo10}
\row{\mattwo{1+p}001}{\cvectwo p0,\cvectwo01}{\mattwo p000}{\cvectwo
		p0}{\cvectwo00,\cvectwo01}
\row{\mattwo1011}{\cvectwo p0,\cvectwo01}{\mattwo p000}{\cvectwo
		p0}{\cvectwo00,\cvectwo01}
\RowWithNote{\mattwo1p11}{\cvectwo p0}{\mattwo z000}{
		\cvectwo p0}{\cvectwo00}{(z=p,2p)}
\splitrow{\mattwo1{\varepsilon p}11}{
	\quad \gen{\cvectwo p0} \quad}{
	\mattwo0000}{\mattwo p000}{\cvectwo00}{\cvectwo p0}{
	\cvectwo00,\cvectwo p0}{\cvectwo00}
\row[(\S)]{\matjordanone}{\cvecthree100,
		\cvecthree001}{\matzero}{\cvecthree000}{\cvecthree000}
		[2ex]
\splitrow{\matjordantwo}{\quad \gen{\cvecthree100}\quad}{
	\matthree001000000}{%
	\matzero}{\cvecthree100}{\cvecthree000}{\cvecthree000}{%
	\cvecthree000,\cvecthree100}
\cline{2-6}
\end{array}
$$
\caption{For each $\tau$, enough choices for $v$}
\label{table:v-choices}
\end{table}

\section{Classification}
\label{sec:classification}

From now on, in writing an extension type
$(N, p ,\tau, v)$ we will feel free to omit $p$,
since it is the same for all types we are considering.
We will also omit $N$, since we can infer it from $\tau$.
That is, $N$ is either $C_{p^2} \times C_p$
or $C_p\times C_p \times C_p$ according as
$\tau$ is a $2$-by-$2$ or $3$-by-$3$ matrix.
Thus, we will refer to a pair $(\tau, v)$ as an extension type.

\begin{mainthm}
Every nonabelian group of order $p^4$ realizes
precisely one of the extension types
$(\tau,v)$
given in Table~\ref{table:distinguish}.
\end{mainthm}

\begin{table}
\newcommand{\newrow}[5][{}]{#1 & #2 & #3 & #4 & \multicolumn{2}{c}{#5} \\}
\newcommand{\newSplitRow}[6][{}]{#1 & #2 & #3 & #4 & #5 & #6 \\}
$$
\begin{array}{c c c |  c c c}
&&\multicolumn{1}{c}{}&&\multicolumn{2}{c}{\text{\# elements of order $\leq p$}} \\
&\tau & \multicolumn{1}{c}{v} & \text{center} & p = 3 & p > 3 \\
\cline{2-6}
\newrow{\mattwo1p01}{\cvectwo00}{C_{p^2}}{p^3}
\newrow{\mattwo1p01}{\cvectwo10}{C_{p^2}}{p^2}
\newrow[(*)]{\mattwo{1+p}001}{\cvectwo00}{C_p\times C_p}{p^3}
\newrow{\mattwo{1+p}001}{\cvectwo01}{C_p\times C_p}{p^2}
\newrow[(*)]{\mattwo1011}{\cvectwo00}{C_p\times C_p}{p^3}
\newrow[(**)]{\mattwo1p11}{\cvectwo00}{C_p}{p^3}
\newSplitRow[(**)]%
	{\mattwo 1{\varepsilon p}11}{\cvectwo00}{C_p}{p^4-p^3+p^2}{
	\phantom{p^4+{}} p^3 \phantom{{}+ p^2}  }
\newSplitRow{\mattwo1{\varepsilon p}11}{\cvectwo p0}{C_p}{p^2}{%
	\text{N/A} }
\newrow{\matjordanone}{\cvecthree000}{C_p\times C_p}{p^4}
\newSplitRow{\matjordantwo}{\cvecthree000}{C_p}{2p^3-p^2}{p^4}
\newSplitRow[(**)]{\matjordantwo}{\cvecthree100}{C_p}{\text{N/A}}{p^3}
\cline{2-6}
\end{array}
$$
\caption{Extension types for all nonabelian groups of order $p^4$}
\label{table:distinguish}
\end{table}

It is understood that if $p=3$, then we ignore
the row in Table~\ref{table:distinguish} containing
``N/A'' in the ``$p=3$'' column, and similarly if $p>3$.
Thus, either way we only consider $10$ rows of the table,
and the theorem is asserting that there are exactly $10$
nonabelian groups of order $p^4$.

\begin{proof}
Considering all choices for $v$ that appear in Table~\ref{table:v-choices},
we obtain eleven pairs $(\tau,v)$,
which are sufficient
for constructing all nonabelian groups of order $p^4$.
All of these pairs appear in
Table~\ref{table:distinguish}
except for
$(\smattwo1011, \scvectwo01)$.
However, in Lemma~\ref{lem:special-equiv}, we will see that
this is equivalent to
$(\smattwo{1+p}001, \scvectwo01)$.
Thus, it only remains to show that the listed extension types
are all pairwise inequivalent.

For each of the groups arising from the pairs $(\tau,v)$
in Table~\ref{table:distinguish},
we compute the isomorphism class of the center,
and we count the elements of order up to $p$.
This alone will distinguish most of these groups from each other.

Suppose $G$ realizes an extension type $(\tau,v)$
from Table~\ref{table:distinguish}.
Then the center of $G$ is precisely the group $N^\tau$ of fixed
points of $\tau$, whose isomorphism
class can be obtained from Table~\ref{table:v-choices}.

Counting the elements of $G$ of order up to $p$ is also straightforward.
First,
note that $G$ is the union of the $p$ cosets $N,Na,\ldots, Na^{p-1}$.
The number of
elements of order up to $p$ in $N$ is easy to compute.
If $N=C_{p^2}\times C_{p}$,
then there are $p^2$ such elements.
If $N=C_{p}\times C_{p}\times C_{p}$,
then there are $p^3$ such elements.
Moreover, for $0<i<p$,
the map $y\mapsto y^i$ induces a bijection between the cosets
$Na$ and $Na^i$ that remains a bijection when restricted to
elements of order $p$.
Thus,
it's enough to count the
elements of order $p$ in the coset $Na$.
For $x\in N$, by Lemma~\ref{lem:power-norm}
we know that
$(xa)^{p}=\NN(x)v$.
If $v\notin \im(\NN)$, then this cannot be the identity,
and so $Na$ has no elements of order $p$.
If $v\in \im(\NN)$,
then the elements of order $p$ in $Na$ are in bijection with
$\ker(\NN)$.
Thus, the number of elements of order up to $p$ in $G$ is
\[
\text{(the number of elements of order $\leq p$ in $N$)}+(p-1)
\begin{cases}
|\ker (\NN)| & \text{if $v\in \im(\NN)$}, \\
0 & \text{otherwise}.
\end{cases}
\]
Note that $|\ker(\NN)| = p^3 / |\im(\NN)|$.
Since $\im(\NN)$ is given 
in Table~\ref{table:v-choices},
we can now count the elements of $G$ of order up to $p$.
The results are presented in
Table~\ref{table:distinguish}.

Inspecting 
Table~\ref{table:distinguish},
we see that the only possible equivalences
are between the two rows labeled ($*$);
and (if $p>3$) among the three rows labeled ($**$).
In Lemma~\ref{lem:special-not-equiv},
we will see that the former equivalence fails.

Consider the three rows labeled ($**$).
Each
determines a group that has precisely $p^3$ elements
of order up to $p$.
In the third group, these must therefore be the elements
of $N$, all of which commute with each other.
However, in the first two groups, it is easy
to find examples of elements of order $p^3$
that do not commute.
Therefore, the third group is distinct from the other two.
In Lemma~\ref{lem:special-not-equiv2},
we will see that the first two groups
are also distinct.

Thus, once we have proved
Lemmata~\ref{lem:special-equiv}, \ref{lem:special-not-equiv},
and~\ref{lem:special-not-equiv2},
the theorem will be established.
\end{proof}

We now take care of unfinished business from the proof.
From now on, let $N = C_{p^2}\times C_{p}$.

\begin{lemma}
\label{lem:special-equiv}
The extension types
$\left( \smattwo{1+p}001, \scvectwo01 \right)$
and
$\left( \smattwo1011,\scvectwo01 \right)$
are equivalent.
\end{lemma}

\begin{proof}
Assume that $G$ realizes the extension type
$(  N,p, \smattwo{1+p}001, \scvectwo01 )$
and is constructed as in the proof of Theorem~\ref{thm:cyclic-extn}.
Let $x$ and $y$ denote generators of $C_{p^2}$ and $C_p$
(respectively) in $N$.
To construct a subgroup $N'$ that is isomorphic to $N$,
let $x' =a^{p-1}$, $y'=x^p$,
and $N'=\langle x',y'\rangle$.
Define $a'=x$ so that $(a')^p=v'=y'$.
Let $\tau'\in \Aut(N')$ 
be the automorphism that acts via conjugation by $a'$.
Consider
\begin{align*}
\tau'(x')
&=a'x'(a')\inv
=x a^{p-1} x\inv
=x \tau^{p-1}(x\inv) a^{p-1} 
=x (x\inv)^{(1+p)^{(p-1)}} a^{p-1} \\
&\quad
=x (x\inv)^{1+(p-1)p} a^{p-1}
=x x^{p-1} a^{p-1}
=x^p a^{p-1}
=y' x'
=x' y' \, ,
\\
\tau'(y')
&=a'y'(a')\inv=xx^px\inv =x^p=y' \, .
\end{align*}
With respect to the generators $x'$ and $y'$,
$\tau'$ thus has matrix representation
$\smattwo1011$,
and $v'$
has vector representation $\scvectwo{0}{1}$.
Therefore,
$G$ realizes the extension type
$\left(N',p,\smattwo1011, \scvectwo01 \right)$.
\end{proof}

\begin{lemma}
\label{lem:special-not-equiv}
The extension types
$\left( \smattwo{1+p}001,\scvectwo00 \right) $
and
$\left( \smattwo1011,\scvectwo00 \right) $
are inequivalent.
\end{lemma}

\begin{proof}
Let $G_1$ and $G_2$, respectively,
realize the given extension types.
For $i=1,2$,
define $H_i=\{g^{p}:g\in G_i\}$.
Using Lemma~\ref{lem:power-norm},
one can compute
that $H_1=H_2 = \langle \scvectwo p0 \rangle$.
It follows that each $H_i$ is a subgroup of the center of $G_i$,
and is thus a normal subgroup of $G_i$.
We will be done if we can show that
$G_1/H_1$ is abelian
and $G_2/H_2$ is not.
To do so, we will consider the image of each automorphism $\tau$
in $\Aut(N/H_i) \cong \Aut(C_p\times C_p) = \GL_2(\F_p)$,
obtained by reducing modulo $p$ the first row of the matrix
for $\tau$.
Since the image of
$\tau=\smattwo{1+p}001$
in $\GL_2(\F_p)$
is the identity,
$\tau$ acts trivially on $N/H_1$.
Hence, $G_1/H_1$ is abelian.
However, $G_2/H_2$ is nonabelian since the image of
$\smattwo1011$
in $\GL_2(\F_p)$ is not the identity.
\end{proof}

If $p=3$, then we are done.  Otherwise,
we still need the following result.

\begin{lemma}
\label{lem:special-not-equiv2}
The extension types
$\left(\smattwo1p11, \scvectwo00 \right)$
and
$\left(\smattwo 1{\varepsilon p}11, \scvectwo00\right)$
are inequivalent.
\end{lemma}

\begin{proof}
Let $G$ be a group determined by the first extension type.
As in the proof of Lemma~\ref{lem:special-equiv},
we have generators $x$, $y$, and $a$ for $G$,
where $x$ and $y$ generate $N$,
and the conjugation action $\tau$ of $a$ on $N$ is
represented by the matrix
$\smattwo1p11$
with respect to our given generators for $N$.
In order for $G$ to also realize the second extension type,
we must have generators $x'$, $y'$, and $a'$ that satisfy
the same relations as $x$, $y$, and $a$, except that
the conjugation action $\tau'$ of $a'$ on $N':=\langle x',y'\rangle$
should now be represented by the matrix
$\smattwo 1{\varepsilon p}11$ with respect to
our given generators for $N'$.

For a contradiction, suppose that there exist such $x'$, $y'$, and $a'$.

Let $Z$ and $H$ denote the center and commutator subgroup,
respectively, of $G$.
It is straightforward to compute that
$Z = \langle \scvectwo p0 \rangle$,
and $H =\langle \scvectwo p0, \scvectwo01\rangle$.
Since $y$ is an element of order $p$ that lies in $H$
but not in $Z$, the same must be true for $y'$.
Since $x$ commutes with $y$, $x'$ must commute with $y'$.
It is straighforward to compute that for $n\in N$
and $k\not\equiv 0 \bmod p$,
$(na^k) y' (na^k)\inv \neq y'$ for $y'\in H \smallsetminus Z$.
Therefore, $x'\in N$, and so $N'=N$.

For any $n\in N$, $na'$ and $a'$ induce the
same conjugation action on $N$.
Therefore, we may assume that $a' = a^k$
for some $0<k<p$.

For any homomorphism $\varphi \colon N \longrightarrow N$,
consider the homomorphism $\varphi -I$
that takes an element $n$ to $\varphi(n) n\inv$.
To obtain the matrix representation of $\varphi-I$
with respect to some set of generators,
take the matrix representation of $\varphi$ and subtract the identity matrix.
For example, with respect to the generators $x$ and $y$,
we see from \eqref{eqn:power-one}
that $(\tau^k -I )$ is represented by the
matrix
$\smattwo{\binom{k}{2}p}{kp}{k}{0}$.
Thus, the composition $(\tau^k-I)^2$ of this map
with itself is represented by the matrix
$\smattwo{k^2p}000$.
Similarly, with respect to the generators $x'$ and $y'$,
$(\tau' - I)$ is represented by the matrix
$\smattwo{\varepsilon p}000$.

Since $x'$ has order $p^2$, we must have
$x'= x^c y^d$ for some $c$ and $d$, with $c\not\equiv0\bmod p$.
Note that $(x')^p = (x^p)^c$.

For any two group elements $g$ and $h$,
let $[g,h]$ denote $ghg\inv h\inv$.
For example, if $n\in N$, then
$[a,n] = \tau(n) n\inv = (\tau-I)(n)$.

We now compute $\bigl[a',[a',x']\bigr]$ in two ways.
First,
$$
\bigl[a',[a',x']\bigr]
= (\tau' - I)^2 (x') = {x'}^{\varepsilon p} = x^{\varepsilon cp}.
$$
Second,
$$
\bigl[a',[a',x']\bigr]
= \bigl[a^k,[a^k,x']\bigr]
= (\tau^k-I)^2(x^c y^d)
= x^{k^2 c p}.
$$
But the results of these two computations cannot be equal,
since $k^2 \not\equiv \varepsilon \bmod p$,
the latter being a nonsquare.
\end{proof}

From the information we have accumulated, it is possible
to determine which nonabelian groups of order $p^4$
can be decomposed into semidirect products of smaller groups,
and which cannot.

\exerdoit

\section{Comments on the $p=2$ case}
\label{sec:p=2}
The Main Theorem
is only valid for $p$ odd.
Of course, one can find the classification
of groups of order $2^4$ in \cite{Wild}.
However, if the reader wants to adapt the machinery
we have used, here is what is required.

The first of our results that depends on $p$ being odd
is Proposition~\ref{prop:no-cyclic},
which says that it is enough to consider cyclic
extensions of just two abelian groups of order $p^3$.
When $p=2$, then one can prove (or find in \cite{Wild})
an analogous result,
but the two groups in question are now
$C_{p^3}$ and $C_{p^2} \times C_p$.
Our analysis of the nonabelian cyclic extensions of
$C_p \times C_p \times C_p$
is thus unnecessary when $p=2$.
Instead, one needs to study the nonabelian extensions
of $C_{p^3}$, imitating the arguments
of Proposition~\ref{prop:no-cyclic}.

The method of analysis of the extensions of $C_{p^2} \times C_p$
remains valid.
However,
some of the calculations that go into
Table~\ref{table:v-choices} yield different answers,
and Lemma~\ref{lem:tau-choices2}
needs to take into account the fact that $\F_2$
contains no nonsquare element,
so there are no analogues of the extension types involving $\varepsilon$.

In light of the above comments,
it is interesting to
classify the groups of order $16$ using as little effort as possible.

\exerdoit

\bibliographystyle{amsplain}
\bibliography{p4biblio}

\providecommand{\bysame}{\leavevmode\hbox to3em{\hrulefill}\thinspace}
\providecommand{\MR}{\relax\ifhmode\unskip\space\fi MR }
\providecommand{\MRhref}[2]{%
  \href{http://www.ams.org/mathscinet-getitem?mr=#1}{#2}
}
\providecommand{\href}[2]{#2}
\begin{thebibliography}{1}

\bibitem{besch-eick-obrien:millenium}
Hans~Ulrich Besche, Bettina Eick, and Eamonn~A. O'Brien, \emph{A millennium
  project: constructing small groups}, Internat. J. Algebra Comput. \textbf{12}
  (2002), no.~5, 623--644. \MR{MR1935567 (2003h:20042)}

\bibitem{Burnside}
William Burnside, \emph{Theory of groups of finite order}, 2nd ed., Dover
  Publications Inc., New York, 1955. \MR{MR0069818 (16,1086c)}

\bibitem{Dummit}
David~S. Dummit and Richard~M. Foote, \emph{Abstract algebra}, 3rd ed., Wiley,
  New York, 2004.

\bibitem{Garlow}
Michael Garlow, \emph{An elementary classification of the groups of order
  $81$}, Master's thesis, The University of Akron, 2006.

\bibitem{Holder}
Otto H{\"o}lder, \emph{Die {G}ruppen der {O}rdnungen {$p\sp 3$}, {$pq\sp 2$},
  {$pqr$}, {$p\sp 4$}}, Math. Ann. \textbf{43} (1893), no.~2-3, 301--412.
  \MR{MR1510814}

\bibitem{obrien-vlee:p7}
Eamonn~A. O'Brien and Michael~R. Vaughan-Lee, \emph{The groups with order
  {$p\sp 7$} for odd prime {$p$}}, J. Algebra \textbf{292} (2005), no.~1,
  243--258. \MR{MR2166803 (2006d:20038)}

\bibitem{Wild}
Marcel Wild, \emph{The groups of order sixteen made easy}, Amer. Math. Monthly
  \textbf{112} (2005), no.~1, 20--31. \MR{MR2110109}

\bibitem{Young}
J.~W.~A. Young, \emph{On the determination of groups whose order is a power of
  a prime}, Amer. J. Math. \textbf{15} (1893), no.~2, 124--178.

\end{thebibliography}

\end{document}